\documentclass[12pt,leqno]{amsart}
\usepackage{amsmath,amsfonts,amssymb,amsthm}
\usepackage{graphicx}
\graphicspath{ }
\usepackage{wrapfig}
\usepackage{accents}
\usepackage{caption}
\usepackage{subcaption}
\usepackage{calligra}

\numberwithin{figure}{section}

\theoremstyle{plain}
\newtheorem{theorem}{Theorem}[section]

\newtheorem{cor}{Corollary}[theorem]
\theoremstyle{definition}

\theoremstyle{remark}

\usepackage{mathtools}
\title[$m$-quasi Einstein manifolds with convex potential]{$m$-quasi Einstein manifolds with convex potential}
\author[A. A. Shaikh, P. Mandal, C. K. Mondal, A. Ali]{Absos Ali Shaikh$^{1*}$, Prosenjit Mandal$^2$, Chandan Kumar Mondal$^{3\dag}$, Akram Ali $^4$}

\address{\noindent\newline $^{1,2,3}$Department of Mathematics,\newline The University of Burdwan,Golapbag,\newline Purba Bardhaman-713101,\newline West Bengal, India}

\address{\noindent\newline $^\dag$School of Sciences,\newline Netaji Subhas Open University,\newline Durgapur Regional Center, Durgapur-713214\newline Paschim Bardhaman,\newline West Bengal, India}

\address{\noindent\newline $^4$Department of Mathematics,\newline  College of Science,\newline King Khalid University,\newline  9004 Abha, Saudi Arabia}

\email{$^1$aask2003@yahoo.co.in, aashaikh@math.buruniv.ac.in}
\email{$^2$prosenjitmandal235@gmail.com}
\email{$^3$chan.alge@gmail.com}
\email{$^4$ akramali133@gmail.com}

\begin{document}
\begin{abstract}
The main objective of this paper is to investigate the $m$-quasi Einstein manifold when the potential function becomes convex. In this article, it is proved that an $m$-quasi Einstein manifold satisfying some integral conditions with vanishing Ricci curvature along the direction of potential vector field has constant scalar curvature and hence the manifold turns out to be an Einstein manifold. It is also shown that in an $m$-quasi Einstein manifold the potential function agrees with  Hodge-de Rham potential up to a constant. Finally, it is proved that if a complete non-compact and non-expanding $m$-quasi Einstein manifold has bounded scalar curvature and the potential vector field has global finite norm, then the scalar curvature vanishes.
\end{abstract}
\noindent\footnotetext{$\mathbf{2020}$\hspace{5pt}Mathematics\; Subject\; Classification: 53C20; 53C21; 53C44. }
\maketitle
\section{Introduction and preliminaries}
A Riemannian manifold $(M,g)$ of dimension $n(> 2)$ is called an $m$-quasi Einstein manifold (see e.g. \cite{BR12,CSW11,SHP21}), if its Ricci tensor $Ric$ satisfies the following relation:
\begin{equation}\label{q1}
Ric+\frac{1}{2}\pounds_Xg-\frac{1}{m}X^\flat\otimes X^\flat=\lambda g,
\end{equation}
where $X$ is a smooth vector field on $M$, $X^\flat$ is the dual 1-form of $X$ with respect to the metric $g$ and $m, \lambda$ are scalar such that $0<m\leq\infty$. If we omit the lie derivative term in (\ref{q1}), then we get the notion of quasi Einstein manifold (see e.g. \cite{DDVV99,SYH09, SKH11} and also references therein) which appeared in the literature while considering the investigation of exact solution of Einstein field equation and also during the consideration of quasiumbilical hypersurfaces, studied by Cartan. In \cite{DVY94} it is shown that $3$-dimensional Cartan hypersurfaces are quasi-Einstein manifolds. Throughout the paper we will consider the $m$-quasi Einstein manifolds (\cite{CSW11}), which is a generalization of Ricci soliton in case of $m$-Bakry-Emery tensor. An $m$-quasi Einstein manifold is said to be expanding (resp., steady or shrinking), if $\lambda<0$ (resp., $ \lambda =0$ or $\lambda>0$). If $X$ is the gradient of a smooth function $f$, then ($\ref{q1}$) reduces to the following form:
\begin{equation}\label{e0}
Ric+\nabla^2f-\frac{1}{m}df\otimes df=\lambda g.
\end{equation}
The function $f$ is called potential function of the $m$-quasi Einstein manifold. The tensor of the left hand side of (\ref{e0}) is called the $m$-Bakry-Emery tensor. Further, if $m=\infty$, then (\ref{e0}) reduces to the equation of gradient Ricci soliton (see e.g. \cite{CK04, HA82,CA20,AM21,SMM20}) with the potential function $f$. If $m$ and $\lambda$ are smooth functions on $M$, then (\ref{e0}) is called the generalized $m$-quasi Einstein manifold \cite{CAT12}. Moreover, taking the trace of (\ref{e0}), we deduce
\begin{equation}\label{e1}
\Delta f=\lambda n-R+\frac{1}{m} |\nabla f|^2,
\end{equation}
where $R$ denotes the scalar curvature of the manifold. Case et al. \cite{CSW11} proved that a compact $m$-quasi Einstein metric with constant scalar curvature is trivial. They also showed that all $2$-dimensional $m$-quasi Einstein metrics on compact manifolds are trivial. In 2012, Barros and Ribeiro \cite{BR12} proved that in a complete non-compact $m$-quasi Einstein manifold if $\lambda n\geq R$ and $|\nabla f|\in L^1(M)$, then the manifold becomes Einstein. In 2015, Hu et al. \cite{HLX15} classified $m$-quasi Einstein manifolds  with parallel Ricci tensor, and also in \cite{HLZ17}, they classified $m$-quasi Einstein manifolds with all eigenvalues of the Ricci tensor as constant. Kim and Shin \cite{KS19} also classified 3-dimensional $m$-quasi Einstein manifolds. Barros and Gomes \cite{BG16} studied compact $m$-quasi Einstein manifolds and also showed that if such a manifold is Einstein, then its potential vector field vanishes. We note that the investigation of various structures of Riemannian manifolds through the convexity revel several geometrical and topological properties of such manifolds. Hence by motivating the above studies as well as the study of \cite{ABR2011}, in the present paper, we prove the following:
\begin{theorem}\label{th1}
Suppose $(M,g)$ is an $m$-quasi Einstein manifold which is complete and non-compact and endowed with a positive convex potential function $f$ and $Ric\geq -(n-1)K$ for some positive constant $K$. If the Ricci curvature of the manifold vanishes along $\nabla f$ and $f$ satisfies
\begin{equation}\label{e4}
\quad \int_{M-B(p,r)}\frac{f}{d(x,p)^2}<\infty,
\end{equation}
where $B(p,r)$ is an open ball with center $p$ and radius $r$, and $d(x,p)$ is the shortest distance between $x$ and $p$ in $M$, then the following holds:\\
(i) The manifold is Einstein,\\
(ii) The scalar curvature of the manifold is a non-positive constant,\\
(iii) The Ricci curvature of the manifold is non-positive everywhere in $M$.
\end{theorem} 
\begin{cor}
Let $(M,g)$ be a complete non-compact $m$-quasi Einstein manifold with positive convex potential function $f$. If the manifold is radially Ricci flat and satisfies the condition (\ref{e4}), then the manifold is Einstein with non-positive Ricci curvature.
\end{cor}
\begin{theorem}\label{th3}
If $(M,g)$ is a compact oriented $m$-quasi Einstein manifold such that  $f$ is a potential function on $M$, then $f$ agrees, upto a constant, with the Hodge-de Rham potential.
\end{theorem}
\begin{theorem}\label{th4}
Let $(M,g)$ be a complete non-compact and non-expanding $m$-quasi Einstein manifold with finite volume. If the potential vector field $X$ is of finite global norm, then there exists an open ball where the scalar curvature $R\geq \lambda n$. 
\end{theorem}
\begin{cor}\label{co1}
Let $(M,g)$ be a complete non-compact and non-expanding $m$-quasi Einstein manifold with finite volume. If the potential vector field $X$ is of finite global norm and the scalar curvature is bounded by  $\lambda n$, then the scalar curvature vanishes in $M$. 
\end{cor}
\section{Proof of the results}
To prove Theorem \ref{th1} we need the following result of Hu et al. \cite{HLX15}, which we state at first.
\begin{theorem}\cite{HLX15}\label{th2}
In an $m$-quasi Einstein manifold $(M,g)$ with constant scalar curvature, $\lambda\leq 0$.
\end{theorem}
\begin{proof}[\textbf{Proof of Theorem \ref{th1}}]
Since $f\in C^\infty(M)$, ring of smooth functions on $M$, is a non-constant convex function on $M$, it follows that \cite{YA74}, $M$ is non-compact. Let us consider the cut-off function, studied in \cite{CC96}, $\varphi_r\in C^2_0(B(p,2r))$ for $r>0$, where $C^2_0(B(p,2r))$ is a class of second order continuously differentiable functions with compact support and $C$ being a constant, $(B(p,2r))\subseteq M$, $p\in M$ such that
\[ \begin{cases} 
		\varphi_r=1  & \text{ in }B(p,r) \\
	  0\leq \varphi_r\leq 1 &\text{ in }B(p,2r)\\      
      \Delta \varphi_r\leq \frac{C}{r^2} &  \text{ in }B(p,2r).
   \end{cases}
\]
Then for $r\rightarrow\infty$, we have $\Delta \varphi^2_r\rightarrow 0$ as $\Delta \varphi^2_r\leq \frac{C}{r^2}$.
The convexity of $f$ implies that $f$ is also subharmonic \cite{GH71}, i.e., $\Delta f\geq 0$. In view of integration by parts, we obtain
\begin{equation}\label{r5}
\int_M \Delta f\varphi^2_r=\int_M f\Delta \varphi^2_r.
\end{equation}
 Since $\varphi_r\equiv 1$ in $B(p,r)$, using (\ref{r5}), we get
 \begin{equation*}
 \int_{B(p,r)}\Delta f=0.
 \end{equation*}
 Again, in view of integration by parts and also by our assumption, we obtain
 \begin{equation*}
0\leq  \int_{B(p,2r)}\varphi_r^2\Delta f=\int_{B(p,2r)-B(p,r)}f\Delta \varphi_r^2\leq \int_{B(p,2r)-B(p,r)}f\frac{C}{r^2}.
 \end{equation*}
 But the right hand side tends to zero as $r\rightarrow \infty$. Hence we get
 $$\int_M\Delta f=0.$$ 
 Thus, the subharmonocity of $f$ implies that $\Delta f=0$ in $M$. Therefore, (\ref{e1}) entails that
 \begin{equation}
 R-\lambda n=\frac{1}{m}|\nabla f|^2,
 \end{equation}
 which implies that $R\geq \lambda n$. Taking $u=\log f$ and then simplifying we obtain  
 $$-\Delta u=|\nabla u|^2=\frac{|\nabla f|^2}{f^2}.$$
 Let $l(x)$ be the distance of $x\in M$ from the fixed point $p$. For any $r>0$, consider the function $\eta:[0,\infty)\rightarrow[0,1]$ satisfying the following properties:
 \[\begin{cases} 
 \eta(t)=1 \text{ for }t\leq r\\
 \eta(t)=0 \text{ for }t\geq 2r\\
 	 \eta'\leq 0\\
      (\eta')^2\leq \frac{(\eta')^2}{\eta}\leq \frac{C}{r^2}\\
      |\eta''|\leq\frac{C}{r^2},
    \end{cases}
 \] 
 for some constant $C<\infty$.
 Now define the function $\eta$ on $M$ by $\eta(x)=\eta(l(x))$ for $x\in M$. Then the function $u$ satisfies the following inequality:
 $$|\nabla\nabla u|^2\geq \frac{1}{n}(\Delta u) ^2=\frac{1}{n}|\nabla u|^4.$$
 By virtue of above inequality and the Bochner formula, we obtain
 \begin{eqnarray*}
 \frac{1}{2}\Delta ( \eta|\nabla u|^2)&=& \frac{1}{2}\Delta \eta |\nabla u|^2+\nabla \eta\nabla|\nabla u|^2+ \frac{1}{2}\eta\Delta|\nabla u|^2\\
 &=& \frac{1}{2}\Delta \eta |\nabla u|^2+\nabla \eta\nabla|\nabla u|^2+\eta|\nabla\nabla u|^2+\eta\nabla u\nabla(\Delta u)+\eta Ric(\nabla u,\nabla u)\\
 &=& \frac{1}{2}\Delta \eta |\nabla u|^2+\eta\nabla \eta\nabla|\nabla u|^2+\eta|\nabla\nabla u|^2-\eta\nabla u\nabla|\nabla u|^2+\eta Ric(\nabla u,\nabla u) \\
 &\geq & \Big(\frac{1}{2\eta}\Delta \eta-\frac{1}{\eta^2}|\nabla\eta|^2+\frac{1}{\eta}\nabla\eta\nabla u \Big)\eta|\nabla u|^2+\Big(\frac{1}{\eta}\nabla\eta-\nabla u \Big)\nabla(\eta|\nabla u|^2)\\
 &&+\frac{1}{n}\eta|\nabla u|^4+\eta Ric(\nabla u,\nabla u).
 \end{eqnarray*}
 Again, calculation shows that $\eta$ satisfies the following inequality:
 \begin{equation}\label{e2}
 |\nabla \eta\nabla u|\leq \frac{\eta|\nabla u|^2}{2n}+\frac{n}{2\eta}|\nabla \eta|^2.
 \end{equation}
 Since $|\nabla u|^2$ is non-zero, there is a point at which $\eta|\nabla u|^2$ is maximum where $\eta$ is smooth, and hence $\Delta(\eta|\nabla u|^2)\leq 0$ and $\nabla(\eta|\nabla u|^2)=0$ such that
 $$0\geq \frac{1}{2}\Delta \eta-\frac{1}{\eta}|\nabla\eta|^2+\nabla\eta\nabla u+\frac{1}{n}\eta|\nabla u|^2+\eta^2 Ric(\nabla u,\nabla u).$$
 By using (\ref{e2}), we have
 \begin{eqnarray}
 -\Delta\eta+\frac{n}{\eta}|\nabla \eta|^2+\frac{2}{\eta}|\nabla \eta|^2&=& -\Delta\eta+\frac{n}{\eta}|\nabla \eta|^2+\frac{2}{\eta}|\nabla \eta|^2+2\nabla\eta\nabla u-2\nabla\eta\nabla u\\
 &\geq & \frac{n}{\eta}|\nabla \eta|^2+2\nabla\eta\nabla u+\frac{2}{n}\eta|\nabla u|^2+2\eta^2 Ric(\nabla u,\nabla u)\\
 &\geq & \frac{1}{n}\eta|\nabla u|^2+2\eta^2 Ric(\nabla u,\nabla u).
 \end{eqnarray}
 Therefore, we get
$$-\eta'\Delta d-\eta''+\frac{n+2}{\eta}(\eta')^2\geq \frac{1}{n}\eta|\nabla u|^2+2\eta^2 Ric(\nabla u,\nabla u).$$ 
Since $Ric\geq -(n-1)K$, the Laplace comparison theorem implies that
$$\Delta d\leq\frac{n-1}{d}(1+\sqrt{K}d)\leq \frac{n-1}{r}+(n-1)\sqrt{K},$$
where $\eta'\neq 0$. Hence, we obtain
$$\frac{1}{n}\eta|\nabla u|^2+2\eta^2 Ric(\nabla u,\nabla u)\leq \frac{C'(n)}{r^2}+\frac{C''(n)}{r}\sqrt{K},$$
where $C'(n)$ and $C''(n)$ are positive constants depend only on $n$. Consequently, we get
$$\frac{1}{n}\eta|\nabla u|^2\leq  \frac{C'(n)}{r^2}+\frac{C''(n)}{r}\sqrt{K}-2\eta^2 Ric(\nabla u,\nabla u).$$
Taking limit as $r\rightarrow\infty$ in both sides, yields for all $x\in M$,
\begin{eqnarray}\label{e3}
\nonumber|\nabla u|^2_x&\leq & -2nRic_x(\nabla u,\nabla u)\\
&=& \frac{-2n}{f^2}Ric_x(\nabla f,\nabla f).
\end{eqnarray}
 But Ricci curvature vanishes along $\nabla f$, i.e., $Ric(\nabla f,\nabla f)= 0$, hence, we have $|\nabla u|=0$, which shows that $f$ is constant. Therefore, (\ref{e0}) implies that the manifold becomes Einstein and $R$ is equal to $\lambda n$ which is a constant. Again, Theorem \ref{th2} implies that in an $m$-quasi Einstein manifold $\lambda\leq 0$. Hence, $R=\lambda n\leq 0$ and it also implies that Ricci curvature is non-positive.
\end{proof}
\begin{proof}[\textbf{Proof of Theorem \ref{th3}}]
Let $(M,g)$ is a Riemannian manifold which is compact and oriented. If $X$ is a vector field on $M$, the by virtue of Hodge-de Rham decomposition theorem, (see e.g. \cite{WF1983}), $X$ can be written as
\begin{equation}\label{hd1}
X=W+\nabla \xi,
\end{equation}
where $div\ W=0$ and $\xi$ is a smooth function called the Hodge-de Rham potential. We consider an $m$-quasi Einstein manifold $(M,g)$ such that equation (\ref{q1}) yields
\begin{equation}\label{qe1}
R+div X-\frac{1}{m}|X|^2=\lambda n.
\end{equation}
Hence ($\ref{hd1}$) entails that $divX=\Delta \xi$ and consequently $(\ref{qe1})$ takes the form
\begin{equation}\label{qe2}
R+\Delta \xi-\frac{1}{m}|X|^2=\lambda n.
\end{equation}
Again from $(\ref{e0})$ we have
 \begin{equation*}\label{qe3}
 R+\Delta f-\frac{1}{m}|\nabla f|^2=\lambda n.
 \end{equation*}
 From $(\ref{qe2})$, it follows that 
 \begin{equation}
 \Delta(f-\xi)=0,
 \end{equation}
 which implies that $f=\xi+C_1$, for some constant $C_1$, this proves the result.
\end{proof}
Let $(M,g)$ be an oriented Riemannian manifold and $\Lambda^k(M)$ be the set of all $k$-th differential forms in $M$. Now for any integer $k\geq 0$, the global inner product in $\Lambda^k(M)$ is defined by
$$\langle \eta,\omega\rangle=\int_M \eta\wedge *\omega,$$
for $\eta,\omega\in\Lambda^k(M)$. Here $`*'$ is the the Hodge star operator. Hence, we define the global norm of $\eta\in\Lambda^k(M)$ by $\|\eta\|^2=\langle \eta,\eta\rangle$ and remark that $\|\eta\|^2\leq \infty$.
\begin{proof}[\textbf{Proof of Theorem \ref{th4}}]
For any $r>0$ we have
\begin{eqnarray}
\nonumber\frac{1}{r}\int_{B(p,2r)}|X|dV & \leq & \Big(\int_{B(p,2r)}\langle X,X\rangle dV \Big)^{1/2} \Big(\int_{B(p,2r)}\Big(\frac{1}{r} \Big)^2 dV\Big)^{1/2}\\
\nonumber&\leq& \|X\|_{B(p,2r)}\frac{1}{r}\Big(Vol(M)\Big)^{1/2},
\end{eqnarray}
where $Vol(M)$ denotes the volume of $M$. Thus we obtain
\begin{equation*}
\liminf_{\ r\rightarrow \infty} \frac{1}{r}\int_{B(p,2r)}|X|dV=0.
\end{equation*}
Again, there exists a Lipschitz continuous function $\omega_r$ such that for some constant $K>0$, (see \cite{YA76}),
\begin{eqnarray*}
&&|d\omega_r|\leq \frac{K}{r}\qquad \text{almost everywhere on }M\\
&&0\leq \omega_r(x)\leq 1\quad\forall x\in M\\
&&\omega_r(x)=1\quad\forall x\in B(p,r)\\
&& \text{ supp }\omega_r\subset B(p,2r).
\end{eqnarray*}
Then taking limit, we get $\lim\limits_{r\rightarrow\infty}\omega_r=1$.
Therefore, by using the function $\omega_r$, we have
\begin{equation*}
\Big|\int_{B(p,2r)}\omega_r div X dV \Big|\leq \frac{C}{r}\int_{B(p,2r)}|X|dV.
\end{equation*}
In view of the $m$-quasi Einstein manifold equation $(\ref{q1})$, we get
\begin{equation*}
\int_M \{\lambda n-R+\frac{1}{m} |X|^2\}dV=0,
\end{equation*}
which yields 
\begin{equation}\label{e5}
\int_M(\lambda n-R)dV\leq 0.
\end{equation}
 Since $R$ is continuous, there is an open ball where $R\geq \lambda n$.
\end{proof}
\begin{proof}[\textbf{Proof of Corollary \ref{co1}}]
Our assumption and (\ref{e5}) together imply that $R=\lambda n$ in $M$. Again, using Theorem \ref{th2}, we get $\lambda\leq 0$. But the manifold is non-expanding $m$-quasi Einstein. Therefore, $R=\lambda n=0$ in $M$.
\end{proof}
\section*{Acknowledgment}
 The authors extend their appreciation to the deanship of scientific research at King Khalid University for funding this work through research groups program under grant number R.G.P.1/50/42.

\end{document}